\documentclass[1 leqno,11pt,psfig]{amsart}
\usepackage{amssymb, amsmath,amsmath,latexsym,amssymb,amsfonts,amsbsy, amsthm,mathtools,graphicx,CJKutf8,CJKnumb,CJKulem,color}
\usepackage{graphicx,epsfig}
\usepackage{graphicx}
\usepackage{amssymb}
\usepackage{graphicx}
\usepackage{graphicx,epsfig}
\usepackage{amsmath}
\usepackage{mathrsfs}
\usepackage{amssymb}

\usepackage{amssymb,mathrsfs,amsfonts,amsmath,amsbsy,tikz}\usepackage{fontenc}\usepackage{textcomp}

\numberwithin{equation}{section}

\allowdisplaybreaks

\newcommand{\beq}{\begin{equation}}
\newcommand{\eeq}{\end{equation}}
\newcommand{\ben}{\begin{eqnarray}}
\newcommand{\een}{\end{eqnarray}}
\newcommand{\beno}{\begin{eqnarray*}}
\newcommand{\eeno}{\end{eqnarray*}}

\newtheorem{theorem}{Theorem}[section]
\newtheorem{definition}[theorem]{Definition}
\newtheorem{lemma}[theorem]{Lemma}
\newtheorem{proposition}[theorem]{Proposition}

\newtheorem{remark}[theorem]{Remark}
\newtheorem{Theorem}{Theorem}[section]

\newtheorem{Corollary}[Theorem]{Corollary}

\begin{document}
\title[Topological entropy on uniform spaces]{Topological entropy of nonautonomous dynamical systems on uniform spaces}

\author{Hua Shao}
\address{Department of Mathematics, Nanjing University of Aeronautics and Astronautics,
  Nanjing 211106, P. R. China}
\address{Key Laboratory of Mathematical Modelling and High Performance Computing of Air Vehicles, Nanjing
University of Aeronautics and Astronautics, MIIT, Nanjing 211106, P. R. China}
\email{huashao@nuaa.edu.cn}
\date{\today}

\maketitle

\begin{abstract}
In this paper, we focus on some properties, calculations and estimations of topological entropy
for a nonautonomous dynamical system $(X,f_{0,\infty})$ generated by a sequence of continuous self-maps
$f_{0,\infty}=\{f_n\}_{n=0}^{\infty}$ on a compact uniform space $X$. We obtain the relations of 
topological entropy among $(X, f_{0,\infty})$, its $k$-th product system and its $n$-th iteration system. 
We confirm that the entropy of $(X, f_{0,\infty})$ equals to that of $f_{0,\infty}$ restricted
to its non-wandering set provided that $f_{0,\infty}$ is equi-continuous. We prove that the entropy of 
$(X, f_{0,\infty})$ is less than or equal to that of its limit system $(X, f)$ when $f_{0,\infty}$ converges uniformly to $f$.
We show that two topologically equi-semiconjugate systems have the same entropy if the equi-semiconjugacy
is finite-to-one. Finally, we derive the estimations of upper and lower bounds of entropy for an invariant
subsystem of a coupled-expanding system associated with a transition matrix.
\end{abstract}

{\bf  Keywords}: Topological entropy; uniform space; nonautonomous dynamical system; topological conjugacy; coupled-expansion.

{2010 {\bf  Mathematics Subject Classification}}: 37B40, 37B55, 54E15.

\section{Introduction}
Let $X$ be a compact uniform space equipped with a uniform structure $\mathcal{U}$ and $f_{0,\infty}=\{f_n\}_{n=0}^{\infty}$ be
a sequence of continuous self-maps on $X$. The pair $(X,f_{0,\infty})$ is called a nonautonomous dynamical system. If $f_n=f$ for all $n\geq0$, then $(X, f_{0,\infty})$ becomes the classical dynamical system $(X,f)$. For any $x_0\in X$, the positive orbit $\{x_n\}_{n=0}^{\infty}$ of $(X,f_{0,\infty})$ starting from $x_0$ is defined by $x_n=f_0^n(x_0)$, where
$f_0^n=f_{n-1}\circ\cdots\circ f_{0}$ for any $n\geq 1$, and $\{x_n\}_{n=0}^{\infty}$ can be seen as a solution
of the nonautonomous difference equation
\[x_{n+1}=f_n(x_n),\;n\geq0.\]

Uniform structures were first introduced by Weil in \cite{Weil37}.
Recall that a uniform structure on $X$ is a collection $\mathcal{U}$ of subsets of $X\times X$ satisfying that

(i) if $\gamma\in\mathcal{U}$, then $\bigtriangleup_{X}\subset\gamma$;

(ii) if $\eta\in\mathcal{U}$ and $\eta\subset\gamma\subset X\times X$, then $\gamma\in\mathcal{U}$;

(iii) if $\gamma,\eta\in\mathcal{U}$, then $\gamma\cap\eta\in\mathcal{U}$;

(iv) if $\gamma\in\mathcal{U}$, then $\gamma^{-1}\in\mathcal{U}$;

(v) if $\gamma\in\mathcal{U}$, then there exists $\eta\in\mathcal{U}$ such that $\eta\circ\eta\subset\gamma$.

In the above, $\bigtriangleup_{X}=\{(x,x): x\in X\}$ denotes the diagonal in $X\times X$;
$\gamma^{-1}=\{(x, y): (y, x)\in\gamma\}$ is the inverse of $\gamma\in\mathcal{U}$, and
$\gamma$ is said to be symmetric if $\gamma^{-1}=\gamma$; and
$\gamma\circ\eta=\{(x, z) : (x, y)\in\gamma,(y, z)\in\eta \;{\rm for} \;{\rm some}\; y\in X\}$
denotes the composition of $\gamma,\eta\in\mathcal{U}$. Denote $\gamma^n=\underbrace{\gamma\circ\cdots\circ\gamma}_{n}$,
and clearly, $\gamma\subset\gamma^n$ for any $n\geq1$.
A member of $\mathcal{U}$ is called an index or entourage and the pair
$(X,\mathcal{U})$ is called a uniform space. A uniform structure $\mathcal{U}$ is separated if $\bigcap_{\gamma\in\mathcal{U}}\gamma=\bigtriangleup_{X}$. It is known that a uniform space $(X,\mathcal{U})$
is Hausdorff if and only if $\mathcal{U}$ is separated.

The topology induced by the uniformity $\mathcal{U}$ or the uniform topology is the family of all subsets $G$
of $X$ such that for any $x\in G$ there exists an $\gamma\in\mathcal{U}$ such that $\gamma[x]\subset G$,
where $\gamma[x]=\{y\in X: (x, y)\in\gamma\}$. Denote by ${\mathcal{U}}^s$, ${\mathcal{U}}^o$ and ${\mathcal{U}}^{s,o}$
be the set of symmetric, open and symmetric open indices in $\mathcal{U}$, respectively.
Note that ${\mathcal{U}}^{s,o}$ is a base for $\mathcal{U}$ \cite{Kelley55}.
A map $f: X\to X$ is called uniformly continuous if $(f\times f)^{-1}(\gamma)\in\mathcal{U}$ for any $\gamma\in\mathcal{U}$.
It is easy to see that $f$ is uniformly continuous if and only if for any $\gamma\in\mathcal{U}$,
there exists $\eta\in\mathcal{U}$ such that $(f(x),f(y))\in\gamma$ for any $(x, y)\in\eta$ \cite{Sal21},
if and only if it is continuous relative to the uniform topology \cite{Kelley55}.
The dynamics of systems on uniform spaces have been studied by many authors
\cite{Arai18,Cecc13,Cecc21,Sal21,Shah20,Yan16,Wang18,Wu19}.
For example, Arai recently in \cite{Arai18} showed that for a continuous action of an Abelian group on a
second countable Baire Hausdorff uniform space without isolated points, Devaney chaos implies Li-Yorke chaos.

Topological entropy provides a numerical measure for the complexity of dynamical systems.
The definition of classical topological entropy was introduced by Adler, Konhelm and McAndrew using open covers
for a continuous map on a compact topological space in 1965 \cite{Adler65}. Their definition was directly inspired
by the one given by Kolmogorov for measure-theoretic entropy. Later, Bowen gave another definition using separated
and spanning sets for a uniformly continuous map on a general metric space \cite{Bowen71}, and this definition is
equivalent to Adler's definition when the space is compact. Hood generalized Bowen's definition to a uniformly
continuous map on a uniform space, and investigated the relations between the entropy of the original map and
the entropy of an induced map on a quotient space \cite{Hood74}. Following Hood's work, Yan and Zeng proved
the two definitions of topological entropy using open covers and using separated and spanning sets
are equivalent on a compact uniform space, and they investigated the relations between pseudo-orbits and
topological entropy \cite{Yan16}. Recently, Ceccherini-Silberstein and Coornaert
extended the notion of topological entropy to a uniformly continuous group action on a compact uniform space \cite{Cecc21}.

It is worth mentioning that Kolyada and Snoha extended the concept of topological entropy to a nonautonomous
dynamical system generated by a sequence of continuous self-maps on a compact metric space, and they obtained
a series of important properties of it. For example, they proved that topological entropy is an invariant under
topological equi-conjugacy \cite{Kolyada96}. Note that the majority of complex systems in biology, physics and
engineering are driven by sequences of different functions, and thus the study on nonautonomous dynamical systems
is of importance in applications. In addition, the behaviors of nonautonomous dynamical systems are much richer and
sometimes quite different than what are expected from the classical cases. For example, Balibrea and Oprocha constructed
a nonautonomous dynamical system on the interval which has positive topological entropy, but does not exhibit Li-Yorke chaos.
For more information on nonautonomous dynamical systems, the readers are referred to \cite{Bali12,Canovas13,Kawan13,Kawan16,Kolyada96,Kolyada99,Liu20,Sal21,Shao16,Shao20,Shao21,Shi09,Xu18,Zhu12}
and references therein.

Motivated by the above work, we shall study the topological entropy of a nonautonomous
dynamical system generated by a sequence of continuous self-maps on a compact uniform space, concentrating on its properties,
calculations and estimations. The rest of the paper is organized as follows. In Section 2, the definitions of topological entropy using open covers and using separated and spanning sets for $(X,f_{0,\infty})$ are introduced, respectively, and these definitions are proved to be equivalent. Several basic properties and calculations of topological entropy of $(X,f_{0,\infty})$ are investigated in Section 3. In Section 4, the relations of topological entropy between two topologically equi-semiconjugate systems are studied, and particularly, they are equivalent if the equi-semiconjugacy is finite-to-one. By establishing the topological equi-semiconjugacy to a subshift of finite type, the estimations of upper and lower bounds of topological entropy for an invariant subsystem of a coupled-expanding system associated with a transition matrix are obtained in Section 5.

\section{Definitions of topological entropy}

In this section, we first recall the definitions of topological entropy using open covers and using separated sets and spanning sets for $(X,f_{0,\infty})$, respectively, and then prove that these definitions are equivalent.

\subsection{Definition with open cover}
Let ${\mathscr{A}}_{1},\cdots,{\mathscr{A}}_{n},\; \mathscr{A}$  be open covers of $X$. Denote
\[\bigvee_{i=1}^{n}{\mathscr{A}}_{i}=\big\{\bigcap_{i=1}^{n}A_{i}: A_{i}\in{\mathscr{A}}_{i},\  1\leq i\leq n\big\},\]
and
\[f_{i}^{-n}(\mathscr{A})=\{f_{i}^{-n}(A): A\in\mathscr{A}\},\;
\mathscr{A}_{i}^{n}(f_{0,\infty})=\bigvee_{j=0}^{n-1}f_{i}^{-j}({\mathscr{A}}),\]
where
\[f_{i}^{n}=f_{i+n-1}\circ\cdots\circ f_{i},
\;f_{i}^{-n}=(f_{i}^{n})^{-1}=f_{i}^{-1}\circ\cdots\circ f_{i+n-1}^{-1},\;i\geq0,\;n\geq1.\]
Let $\mathcal{N}(\mathscr{A})$ be the minimal possible cardinality of all subcovers chosen from $\mathscr{A}$.
Then the topological entropy of $(X,f_{0,\infty})$ on the cover $\mathscr{A}$ is defined by
\begin{align}\label{de}
h(f_{0, \infty},\mathscr{A})=\limsup_{n\rightarrow\infty}\frac{1}{n}\log\mathcal{N}\big(\mathscr{A}_{0}^{n}(f_{0,\infty})\big),
\end{align}
and the topological entropy of $(X,f_{0,\infty})$ is defined by
\begin{align*}
h_{top}(f_{0, \infty})=\sup\{h(f_{0,\infty},\mathscr{A}): \mathscr{A}\; {\rm is\; an\; open\; cover\; of\; X\;}\}.
\end{align*}
If each element of ${\mathscr{A}}_{2}$ is contained in some member of ${\mathscr{A}}_{1}$, then we say
${\mathscr{A}}_{2}$ is a refinement of ${\mathscr{A}}_{1}$, and denote it as ${\mathscr{A}}_{1}\prec{\mathscr{A}}_{2}$.
Clearly,
\[{\mathscr{A}}_{1}\prec{\mathscr{A}}_{2}\Rightarrow h(f_{0,\infty},\mathscr{A}_{1})\leq h(f_{0,\infty},\mathscr{A}_{2}).\]

Let $\Lambda$ be any nonempty subset (not necessarily compact or invariant) of $X$. Denote the cover $\{A\cap \Lambda : A\in\mathscr{A}\}$ of the set $\Lambda$ by $\mathscr{A}|_{\Lambda}$. Then the topological entropy of $f_{0, \infty}$ on $\Lambda$ is defined by
\begin{align}\label{subset}
h(f_{0,\infty},\Lambda):=\sup\big\{\limsup_{n\to\infty}\frac{1}{n}\log\mathcal{N}\big({\mathscr{A}_{0}^{n}(f_{0,\infty})}|_{\Lambda}\big):
\mathscr{A}\; {\rm is\; an\; open\; cover\; of\; X}\big\}.
\end{align}

\subsection{Definitions with separated sets and spanning sets}
Let $n\geq1$ and $\gamma\in{\mathcal{U}}^{s,o}$.  A set $E\subset X$ is called $(n,\gamma)$-separated if for each pair $x, y\in E$ with $x\neq y$, there exists $0\leq j\leq n-1$ such that $(f_0^j(x),f_0^j(y))\notin\gamma$; a set $F\subset X$ is called $(n,\gamma)$-spans another set $K\subset X$ if for each $x\in K$ there exists $y\in F$ such that $(f_0^j(x),f_0^j(y))\in\gamma$ for any $0\leq j\leq n-1$. For a set $\Lambda\subset X$, let $s_n(f_{0,\infty},\Lambda,\gamma)$ be the maximal cardinality of an $(n,\gamma)$-separated set in $\Lambda$, $r_n(f_{0,\infty},\Lambda,\gamma)$ and $r_n^{X}(f_{0,\infty},\Lambda,\gamma)$ be the minimal cardinality of a set in $\Lambda$ and in $X$, respectively, which $(n,\gamma)$-spans $\Lambda$. Clearly,
\begin{align}\label{520}
r_n^{X}(f_{0,\infty},\Lambda,\gamma)\leq r_n(f_{0,\infty},\Lambda,\gamma).
\end{align}
Since $X$ is compact, $s_n(f_{0,\infty},\Lambda,\gamma)$, $r_n(f_{0,\infty},\Lambda,\gamma)$ and $r_n^{X}(f_{0,\infty},\Lambda,\gamma)$ are finite.

Note that the cardinality of a set $\Lambda\subset X$ is denoted by $|\Lambda|$. We have the following inequalities.

\begin{proposition}\label{1}
Let $\Lambda\subset X$, $\gamma_1,\gamma_2,\gamma,\eta\in\mathcal{U}^{s,o}$ and $n\geq1$.

{\rm(i)}If $\gamma_1\subset\gamma_2$, then
\[s_n(f_{0,\infty},\Lambda,\gamma_1)\geq s_n(f_{0,\infty},\Lambda,\gamma_2),\;r_n(f_{0,\infty},\Lambda,\gamma_1)\geq r_n(f_{0,\infty},\Lambda,\gamma_2).\]

{\rm(ii)} If $\eta^2\subset\gamma$, then
\[r_n(f_{0,\infty},\Lambda,\gamma)\leq s_n(f_{0,\infty},\Lambda,\gamma)\leq r_n^{X}(f_{0,\infty},\Lambda,\eta)\leq r_n(f_{0,\infty},\Lambda,\eta).\]
\end{proposition}

\begin{proof}
It is only to prove (ii) since (i) is directly derived by the definitions.
Since an $(n,\gamma)$-separated set in $\Lambda$ with the maximal cardinality $(n,\gamma)$-spans $\Lambda$,
\begin{align}\label{001}
r_n(f_{0,\infty},\Lambda,\gamma)\leq s_n(f_{0,\infty},\Lambda,\gamma).
\end{align}
Let $E$ be an $(n,\gamma)$-separated set in $\Lambda$ with $|E|=s_n(f_{0,\infty},\Lambda,\gamma)$ and
$F$ be a subset of $X$ which $(n,\eta)$-spans $\Lambda$ with $|F|=r_n^{X}(f_{0,\infty},\Lambda,\eta)$. For any $x\in E$,
there exists $g(x)\in F$ such that $\big(f_0^j(x),f_0^j(g(x))\big)\in\eta$ for all $0\leq j\leq n-1$. This defines a map $g: E\to F$. If $g(x)=g(y)$ for some $x\neq y\in E$, then
\[(f_0^j(x),f_0^j(y))\in\eta^2\subset\gamma,\;0\leq j\leq n-1,\]
which is a contradiction to the fact that $E$ is an $(n,\gamma)$-separated set. So, $g$ is injective and then
\[|E|=|g(E)|\leq|F|.\]
Hence,
\begin{align}\label{002}
s_n(f_{0,\infty},\Lambda,\gamma)\leq r_n^{X}(f_{0,\infty},\Lambda,\eta).
\end{align}
Therefore, the inequality follows from (\ref{520})-(\ref{002}).
\end{proof}

Let $\Lambda\subset X$ and $\gamma\in{\mathcal{U}}^{s,o}$. Denote
\begin{align}\label{12}
\bar{s}(f_{0,\infty},\Lambda,\gamma)=\limsup_{n\to\infty}\frac{1}{n}\log s_n(f_{0,\infty},\Lambda,\gamma),
\bar{r}(f_{0,\infty},\Lambda,\gamma)=\limsup_{n\to\infty}\frac{1}{n}\log r_n(f_{0,\infty},\Lambda,\gamma).
\end{align}
Note that ${\mathcal{U}}^{s,o}$ is a base for $\mathcal{U}$, and it is a directed set under set inclusion. Thus, $\bar{s}(f_{0,\infty},\Lambda,\gamma)$ and $\bar{r}(f_{0,\infty},\Lambda,\gamma)$ are nets in $\mathbf{R^{+}}$. So, $\lim_{\gamma\in{\mathcal{U}}^{s,o}}\bar{s}(f_{0,\infty},\Lambda,\gamma)$ and $\lim_{\gamma\in{\mathcal{U}}^{s,o}}\bar{r}(f_{0,\infty},\Lambda,\gamma)$ exist and they are equal by Proposition \ref{1}. Hence, define the topological entropy of $f_{0,\infty}$ on the set $\Lambda$ as
\begin{align*}
h(f_{0,\infty},\Lambda)=\lim_{\gamma\in{\mathcal{U}}^{s,o}}\bar{r}(f_{0,\infty},\Lambda,\gamma)
=\lim_{\gamma\in{\mathcal{U}}^{s,o}}\bar{s}(f_{0,\infty},\Lambda,\gamma).
\end{align*}
By Proposition \ref{1}, we also have that
\[h(f_{0,\infty},\Lambda)=\lim_{\gamma\in{\mathcal{U}}^{s,o}}\limsup_{n\to\infty}\frac{1}{n}\log r_n^X(f_{0,\infty},\Lambda,\gamma).\]
If $\Lambda=X$, then $h(f_{0,\infty},X)$, briefly write as $h(f_{0,\infty})$, is called the topological entropy of $(X, f_{0,\infty})$.

\subsection{Equivalence of the two definitions}
The following basic result is an extension of the Lebesgue covering lemma to uniform spaces.

\begin{lemma}\label{Lebesgue covering lemma}\cite{Kelley55}
Let $(X,\mathcal{U})$ be a compact uniform space and $\mathscr{A}$ be an open cover of $X$. Then there exists
$\gamma\in{\mathcal{U}}^{s,o}$ such that $\gamma[x]$ is a subset of some member of $\mathscr{A}$ for any $x\in X$.
\end{lemma}

Let $\gamma\in{\mathcal{U}}^{s,o}$. Denote $\mathscr{A}_\gamma:=\{\gamma[x]: x\in X\}$. Then $\mathscr{A}_\gamma$ is an open cover of $X$.

\begin{proposition}\label{2}
Let $\gamma,\eta\in{\mathcal{U}}^{s,o}$ and $n\geq1$.

{\rm(i)}
\[\mathcal{N}\big((\mathscr{A}_\gamma)_{0}^{n}(f_{0,\infty})|_{\Lambda}\big)\leq r_n(f_{0,\infty},\Lambda,\gamma).\]

{\rm(ii)} If $\eta^2\subset\gamma$, then
\[s_n(f_{0,\infty},\Lambda,\gamma)\leq \mathcal{N}\big((\mathscr{A}_\eta)_{0}^{n}(f_{0,\infty})|_{\Lambda}\big).\]

{\rm(iii)}
\[h_{top}(f_{0,\infty},\Lambda)=\lim_{\gamma\in{\mathcal{U}}^{s,o}}
\limsup_{n\to\infty}\frac{1}{n}\log\mathcal{N}\big({({\mathscr{A}_{\gamma}})_{0}^{n}(f_{0,\infty})}|_{\Lambda}\big).\]
\end{proposition}

\begin{proof}
(i) Let $F$ be a subset of $\Lambda$ which $(n,\gamma)$-spans $\Lambda$ with $|F|=r_n(f_{0,\infty},\Lambda,\gamma)$.
Then for any $y\in \Lambda$, there exists $x\in F$ such that $(f_{0}^{j}(x),f_{0}^{j}(y))\in\gamma$ for any $0\leq j\leq n-1$.
Thus, $y\in\bigcap_{j=0}^{n-1}f_{0}^{-j}\big(\gamma[f_{0}^{j}(x)]\big)$. This implies that
\[\Lambda\subset\bigcup_{x\in F}\bigcap_{j=0}^{n-1}f_{0}^{-j}\big(\gamma[f_{0}^{j}(x)]\big).\]
So, $\{\big(\cap_{j=0}^{n-1}f_{0}^{-j}(\gamma[f_{0}^{j}(x)])\big)\bigcap\Lambda:x\in F\}$ is an open cover of $\Lambda$, and it is also a subcover of $(\mathscr{A}_\gamma)_{0}^{n}(f_{0,\infty})|_{\Lambda}$. Hence,
\[\mathcal{N}\big((\mathscr{A}_\gamma)_{0}^{n}(f_{0,\infty})|_{\Lambda}\big)\leq |F|=r_n(f_{0,\infty},\Lambda,\gamma).\]

{\rm(ii)} Let $E$ be an $(n,\gamma)$-separated set in $\Lambda$ with $|E|=s_n(f_{0,\infty},\Lambda,\gamma)$.
Since every member of $(\mathscr{A}_\eta)_{0}^{n}(f_{0,\infty})|_{\Lambda}$ can contain at most one point of $E$,
\[|E|\leq\mathcal{N}\big((\mathscr{A}_\eta)_{0}^{n}(f_{0,\infty})|_{\Lambda}\big).\]
Thus,
\[s_n(f_{0,\infty},\Lambda,\gamma)\leq \mathcal{N}\big((\mathscr{A}_\eta)_{0}^{n}(f_{0,\infty})|_{\Lambda}\big).\]

{\rm(iii)} Let $\mathscr{A}$ be an open cover of $X$. It follows from Lemma \ref{Lebesgue covering lemma} that there exists
$\gamma\in{\mathcal{U}}^{s,o}$ such that $\mathscr{A}\prec\mathscr{A}_\gamma$, which yields that
\[\limsup_{n\to\infty}\frac{1}{n}\log\mathcal{N}\big(\mathscr{A}_{0}^{n}(f_{0,\infty})|_{\Lambda}\big)
\leq\limsup_{n\to\infty}\frac{1}{n}\log\mathcal{N}\big((\mathscr{A}_{\gamma})_{0}^{n}(f_{0,\infty})|_{\Lambda}\big).\]
Since $\mathscr{A}$ is arbitrary,
\[h_{top}(f_{0,\infty},\Lambda)=\lim_{\gamma\in{\mathcal{U}}^{s,o}}\limsup_{n\to\infty}\frac{1}{n}\log\mathcal{N}
\big((\mathscr{A}_{\gamma})_{0}^{n}(f_{0,\infty})|_{\Lambda}\big).\]
\end{proof}

By Proposition \ref{2}, we have the following result:

\begin{theorem}\label{11}
Let $f_{0,\infty}$ be a sequence of continuous self-maps on a compact uniform space $X$ and
$\Lambda$ be a nonempty subset of $X$. Then
\[h_{top}(f_{0,\infty},\Lambda)=h(f_{0,\infty},\Lambda).\]
\end{theorem}

\begin{remark}\label{1e}
{\rm(i)} Below we no longer distinguish between $h_{top}(f_{0,\infty},\Lambda)$ and $h(f_{0,\infty},\Lambda)$,
and uniformly use $h(f_{0,\infty},\Lambda)$ to denote the topological entropy of $f_{0,\infty}$ on the set $\Lambda$.

{\rm(ii)} If $f_{n}=f$ for all $n\geq0$, then the limit in {\rm(\ref{subset})} exists by \cite{Adler65}.
This, together with Proposition \ref{2}, implies that the $``\limsup"$ in {\rm(\ref{12})} can be replaced by
$``\liminf"$.
\end{remark}

\section{some properties and calculations of topological entropy}

In this section, some properties and calculations of topological entropy for $(X,f_{0,\infty})$ are investigated.

First, the relations of topological entropy between $(X,f_{0,\infty})$ and its $n$-th iteration system $(X,f_{0,\infty}^{n})$
are studied, where $f_{0,\infty}^{n}=\{f_{kn}^{n}\}_{k=0}^{\infty}$ and $f_{kn}^{n}=f_{kn+n-1}\circ\cdots\circ f_{kn}$
for each $k\geq0$ and $n\geq1$.

\begin{proposition}\label{21}
Let $f_{0,\infty}$ be a sequence of continuous self-maps on a compact uniform space $(X,\mathcal{U})$. Then
\[h(f_{0,\infty}^{n})\leq nh(f_{0,\infty}),\;n\geq1.\]
Furthermore, if $f_{0,\infty}$ is equi-continuous, then
\[h(f_{0,\infty}^{n})=nh(f_{0,\infty}),\;n\geq1.\]
\end{proposition}

\begin{proof}
Let $n\geq1$. It suffices to prove that $h(f_{0,\infty}^{n})\geq nh(f_{0,\infty})$ when $f_{0,\infty}$ is equi-continuous
by Lemma 4.2 in \cite{Kolyada96}. Since $f_{0,\infty}$ is equi-continuous, for any $\gamma\in{\mathcal{U}}^{s,o}$, there exists
$\eta\in{\mathcal{U}}^{s,o}$ such that for any $x,y\in X$,
\[(x,y)\in \eta\Rightarrow(f_{i}^{k}(x),f_{i}^{k}(y))\in \gamma,\;0\leq k\leq n-1,\;i\geq0.\]
This implies that
\[(f_{i}^{k}(x),f_{i}^{k}(y))\notin \gamma\; {\rm for}\; {\rm some}\; 0\leq k\leq n-1\;{\rm and}\; i\geq0\Rightarrow(x,y)\notin \eta.\]
Thus, if $M$ is an $(mn,\gamma)$-separated set in $X$ under $f_{0,\infty}$ for some $m\geq1$ , then it is also an $(m,\eta)$-separated set in $X$ under $f_{0,\infty}^n$. So,
\[s_{mn}(f_{0,\infty},X,\gamma)\leq s_m(f_{0,\infty}^n,X,\eta).\]
This, together with the fact that
\[s_{(m-1)n+r}(f_{0,\infty},X,\gamma)\leq s_{mn}(f_{0,\infty},X,\gamma),\;r=1,2,\cdots,n,\]
implies that
\[s_{(m-1)n+r}(f_{0,\infty},X,\gamma)\leq s_m(f_{0,\infty}^n,X,\eta),\;r=1,2,\cdots,n.\]
Hence, for any $1\leq r\leq n$,
\[\limsup_{m\to\infty}\frac{1}{m}\log s_m(f_{0,\infty}^n,X,\eta)\geq n\limsup_{m\to\infty}\frac{1}{(m-1)n+r}
\log s_{(m-1)n+r}(f_{0,\infty},X,\gamma),\]
which yields that
\[\lim_{\eta\in{\mathcal{U}}^{s,o}}\limsup_{m\to\infty}\frac{1}{m}\log s_m(f_{0,\infty}^n,X,\eta)\geq n\lim_{\gamma\in{\mathcal{U}}^{s,o}}\limsup_{m\to\infty}\frac{1}{(m-1)n+r}\log s_{(m-1)n+r}(f_{0,\infty},X,\gamma)\]
for any $1\leq r\leq n$, and therefore, $h(f_{0,\infty}^{n})\geq nh(f_{0,\infty})$.
\end{proof}

Given another system $(Y,g_{0,\infty})$, where $Y$ is a compact uniform space equipped with a uniform structure $\mathcal{V}$
and $g_{0,\infty}=\{g_{n}\}_{n=0}^{\infty}$ is a sequence of continuous self-maps on $Y$. Let $X\times Y$ be the product space of $X$ and $Y$, and $f_{0,\infty}\times g_{0,\infty}=\{f_{n}\times g_{n}\}_{n=0}^{\infty}$. Then $f_{n}\times g_{n}$ is a continuous self-map on $X\times Y$
for any $n\geq0$. Note that the family of sets of the form $W(\gamma,\eta)$ is a base for the product uniformity of $X\times Y$, where
\[W(\gamma,\eta)=\{\big((x,y),(u,v)\big): (x,u)\in \gamma \;{\rm and}\; (y,v)\in \eta\},\;\gamma\in{\mathcal{U}}^{s,o}, \;\eta\in{\mathcal{V}}^{s,o}.\]

\begin{proposition}\label{223}
Let $f_{0,\infty}$ and $g_{0,\infty}$ be two sequences of continuous self-maps on compact uniform spaces $(X,\mathcal{U})$ and $(Y,\mathcal{V})$, respectively. Then
\[h(f_{0,\infty}\times g_{0,\infty})\leq h(f_{0,\infty})+h(g_{0,\infty}).\]
\end{proposition}

\begin{proof}
Let $\gamma\in{\mathcal{U}}^{s,o}$, $\eta\in{\mathcal{V}}^{s,o}$ and $n\geq1$.
Suppose that $F_1$ is a set in $X$ which $(n,\gamma)$-spans $X$ under $f_{0,\infty}$ with $|F_1|=r_n(f_{0,\infty},X,\gamma)$,
and $F_2$ is a set in $Y$ which $(n,\eta)$-spans $Y$ under $g_{0,\infty}$ with $|F_2|=r_n(g_{0,\infty},Y,\eta)$. Then
$F_1\times F_2$ is a set in $X\times Y$ which $(n,W(\gamma,\eta))$-spans $X\times Y$ under $f_{0,\infty}\times g_{0,\infty}$.
In fact, for any $(x,y)\in X\times Y$, there exist $z_1\in F_1$ and $z_2\in F_2$ such that $(f_{0}^{k}(x),f_{0}^{k}(z_1))\in \gamma$
and $(g_{0}^{k}(y),g_{0}^{k}(z_2))\in \eta$ for any $0\leq k\leq n-1$, which yields that
\[\big((f_{0}^{k}(x),g_{0}^{k}(y)),(f_{0}^{k}(z_1),g_{0}^{k}(z_2))\big)\in W(\gamma,\eta),\;0\leq k\leq n-1.\]
Then
\[r_n(f_{0,\infty}\times g_{0,\infty},X\times Y,W(\gamma,\eta))\leq r_n(f_{0,\infty},X,\gamma)r_n(g_{0,\infty},Y,\eta).\]
Thus,
\begin{align*}
&\lim_{\gamma\in{\mathcal{U}}^{s,o},\eta\in{\mathcal{V}}^{s,o}}\limsup_{n\to\infty}\frac{1}{n}
\log r_n(f_{0,\infty}\times g_{0,\infty},X\times Y,W(\gamma,\eta))\\
&\leq\lim_{\gamma\in{\mathcal{U}}^{s,o}}\limsup_{n\to\infty}\frac{1}{n}\log r_n(f_{0,\infty},X,\gamma)
+\lim_{\eta\in{\mathcal{V}}^{s,o}}\limsup_{n\to\infty}\frac{1}{n}\log r_n(g_{0,\infty},Y,\eta).
\end{align*}
Hence,
$h(f_{0,\infty}\times g_{0,\infty})\leq h(f_{0,\infty})+h(g_{0,\infty})$.
\end{proof}

If $f_{n}=f$ and $g_{n}=g$ for all $n\geq1$, then ``$\leq$" in Proposition \ref{223} can be replaced by ``$=$".

\begin{Corollary}
Let $f$ and $g$ be two continuous self-maps on compact uniform spaces $(X,\mathcal{U})$ and $(Y,\mathcal{V})$, respectively. Then
\[h(f\times g)=h(f)+h(g).\]
\end{Corollary}

\begin{proof}
It is only to show $h(f\times g)\geq h(f)+h(g)$ by Proposition \ref{223}.
Let $\gamma\in{\mathcal{U}}^{s,o}$, $\eta\in{\mathcal{V}}^{s,o}$ and $n\geq1$. Suppose that $E_1$ is an $(n,\gamma)$-separated set in $X$ under $f$ with $|E_1|=s_n(f,X,\gamma)$, and $E_2$ is an $(n,\eta)$-separated set in $Y$ under $g$ with $|E_2|=s_n(g,Y,\eta)$. It is to show that $E_1\times E_2$ is an $(n,W(\gamma,\eta))$-separated set in $X\times Y$ under $f\times g$. Let $x=(x_1,x_2),y=(y_1,y_2)\in E_1\times E_2$ with $x\neq y$. Without loss of generality, suppose that $x_{1}\neq y_{1}$. Then there exists $0\leq j\leq n-1$ such that $(f^{j}(x_{1}),f^{j}(y_{1}))\notin\gamma$, which implies that
\[\big((f^{j}(x_{1}),g^{j}(x_{2})),(f^{j}(y_{1}),g^{j}(y_{2}))\big)\notin W(\gamma,\eta).\]
Thus,
\[s_n(f\times g,X\times Y,W(\gamma,\eta))\geq s_n(f,X,\gamma)s_n(g,Y,\eta),\]
which yields that
\begin{align*}
&\lim_{\gamma\in{\mathcal{U}}^{s,o},V\in{\mathcal{V}}^{s,o}}\liminf_{n\to\infty}\frac{1}{n}\log s_n(f\times g,X\times Y,W(\gamma,\eta))\\
&\geq \lim_{\gamma\in{\mathcal{U}}^{s,o}}\liminf_{n\to\infty}\frac{1}{n}\log s_n(f,X,\gamma)
+\lim_{\gamma\in{\mathcal{U}}^{s,o}}\liminf_{n\to\infty}\frac{1}{n}\log s_n(g,Y,\eta).
\end{align*}
This, together with Remark \ref{1e} (ii), implies that $h(f\times g)\geq h(f)+h(g)$.
\end{proof}

The following result reveals the relations of topological entropy between $(X,f_{0,\infty})$ and its $k$-th product system $(X^k,f_{0,\infty}^{(k)})$, where $X^k=\underbrace{X\times\cdots\times X}_{k}$, $f_{0,\infty}^{(k)}=\{f_{n}^{(k)}\}_{n=0}^{\infty}$ and
$f_{n}^{(k)}=\underbrace{f_n\times\cdots\times f_n}_{k}$, $k\geq1$.

\begin{proposition}\label{22}
Let $f_{0,\infty}$ be a sequence of continuous self-maps on a compact uniform space $(X,\mathcal{U})$. Then
\[h(f_{0,\infty}^{(k)})=kh(f_{0,\infty}),\;k\geq1.\]
\end{proposition}

\begin{proof}
Let $k\geq1$. By Proposition \ref{223}, $h(f_{0,\infty}^{(k)})\leq kh(f_{0,\infty})$.
On the other hand, let $\eta\in{\mathcal{U}}^{s,o}$ and $E\subset X$ be an $(n,\eta)$-separated set under $f_{0,\infty}$ with $|E|=s_n(f_{0,\infty},X,\eta)$ for any fixed $n\geq1$. Then $E^k\subset X^k$ is an $(n,W(\underbrace{\eta,\cdots,\eta}_{k}))$-separated set under $f_{0,\infty}^{(k)}$. In fact, for any $x=(x_1,\cdots,x_k),y=(y_1,\cdots,y_k)\in E^k$ with $x\neq y$, there exists $1\leq i_0\leq k$
such that $x_{i_0}\neq y_{i_0}$. Then there exists $0\leq j\leq n-1$ such that $(f_{0}^{j}(x_{i_0}),f_{0}^{j}(y_{i_0}))\notin\eta$,
which implies that
\[\big((f_{0}^{j}(x_{1}),\cdots,f_{0}^{j}(x_{k})),(f_{0}^{j}(y_{1}),\cdots,f_{0}^{j}(y_{k}))\big)\notin W(\eta,\cdots,\eta).\]
Thus,
\[s_n\big(f_{0,\infty}^{(k)},X^k,W(\eta,\cdots,\eta)\big)\geq \big(s_n(f_{0,\infty},X^k,\eta)\big)^{k}.\]
So,
\begin{align*}
\lim_{\eta\in{\mathcal{U}}^{s,o}}\limsup_{n\to\infty}\frac{1}{n}\log s_n(f_{0,\infty}^{(k)},X^k,W(\eta,\cdots,\eta))
&\geq k\lim_{\eta\in{\mathcal{U}}^{s,o}}\limsup_{n\to\infty}\frac{1}{n}\log s_n(f_{0,\infty},X,\eta)\\
&=kh(f_{0,\infty}).
\end{align*}
Therefore, $h(f_{0,\infty}^{(k)})\geq kh(f_{0,\infty})$.
\end{proof}

An open set $U\subset X$ is said to be wandering under $(X,f_{0,\infty})$ if $f_{n}^{k}(U)\cap U=\emptyset$ for all $n\geq0$ and $k\geq1$,
and $x\in X$ is called a wandering point if it belongs to some wandering set. Thus, $x\in X$ is called a non-wandering point under $(X,f_{0,\infty})$ if for any open set $U$ containing $x$, there exist $n\geq0$ and $k\geq1$ such that $f_{n}^{k}(U)\cap U\neq\emptyset$.
Denote the set of all non-wandering points of $(X,f_{0,\infty})$ by $\Omega(f_{0,\infty})$.

\begin{theorem}\label{h}
Let $f_{0,\infty}=\{f_n\}_{n=0}^{\infty}$ be a sequence of equi-continuous self-maps on a compact uniform space $(X,\mathcal{U})$. Then
\[h(f_{0,\infty})= h\big(f_{0,\infty},\Omega(f_{0,\infty})\big).\]
\end{theorem}

\begin{proof}

It is only to show that $h(f_{0,\infty})\leq h(f_{0,\infty},\Omega(f_{0,\infty}))$ since ``$\geq$" is trivial.
Let $\mathscr{A}$ be an open cover of $X$. Fix $n\geq1$. Let $\mathscr{B}_n$ be a minimal subcover of $\Omega(f_{0,\infty})$
which is chosen from ${\mathscr{A}}_{0}^{n}(f_{0,\infty})$. Denote $K=X\setminus\cup_{B\in\mathscr{B}_n}B$. Then $K$ is compact and all the points in $K$ are wandering. For any $x\in K$, there exists $A_x\in{\mathscr{A}}_{0}^{n}(f_{0,\infty})$ such that $x\in A_x$. Since $x$ is wandering, there exists an open neighborhood $U_{x}$ of $x$ such that $U_{x}$ is a wandering set and $U_{x}\subset A_x$. The open cover $\{U_x\cap K: x\in K\}$ of $K$ has a finite subcover $\{U_{x_i}\cap K: 1\leq i\leq l\}$ for some $l\geq1$.
Denote $\mathscr{C}_n=\mathscr{B}_n\cup\{U_{x_1}\cap K,\cdots,U_{x_l}\cap K\}$. Then $\mathscr{C}_n$ is an open cover of $X$ and
${\mathscr{A}}_{0}^{n}(f_{0,\infty})\prec\mathscr{C}_n$. For any $k\geq2$, we consider any nonempty element of $(\mathscr{C}_n)_{0}^{k}(f_{0,\infty}^{n})$,  it is of the form
\[C_{0}\cap f_{0}^{-n}(C_{1})\cap\cdots\cap f_{0}^{-n}\circ f_{n}^{-n}\circ\cdots\circ f_{(k-2)n}^{-n}(C_{k-1}),\]
where $C_i\in\mathscr{C}_n$ for any $0\leq i\leq k-1$.
If $C_i=C_j$ for some $i<j$, then
\[\emptyset\neq f_{(j-1)n}^{n}\circ\cdots\circ f_{in}^{n}(C_{i})\cap C_i=f_{in}^{(j-i)n}(C_i)\cap C_i,\]
which implies that $C_i$ is not wandering, and thus $C_i\in\mathscr{B}_n$.
With the same method used in the proof of Lemma 4.1.5 in \cite{Alse93}, one shows that
\[|(\mathscr{C}_n)_{0}^{k}(f_{0,\infty}^{n})|\leq(m+1)!k^m|\mathscr{B}_n|^k,\]
where $|\mathscr{B}_n|$ denotes the number of elements in $\mathscr{B}_n$ and
$m=|\mathscr{C}_n\setminus\mathscr{B}_n|$. Thus,
\[h(f_{0,\infty}^n,\mathscr{C}_n)
=\limsup_{k\to\infty}\frac{1}{k}\log\mathcal{N}\big((\mathscr{C}_n)_{0}^{k}(f_{0,\infty}^{n})\big)
\leq\limsup_{k\to\infty}\frac{1}{k}\log((m+1)!k^m|\mathscr{B}_n|^k)=\log{\mathscr{B}_n}.\]
For any $\epsilon>0$, there exists an open cover $\mathscr{A}$ of $X$ such that
\[h(f_{0,\infty}^n)<h(f_{0,\infty}^n,\mathscr{A})+\epsilon.\]
This, together with Proposition \ref{21}, implies that
\begin{align*}
h(f_{0,\infty})&=\frac{1}{n}h(f_{0,\infty}^n)<\frac{1}{n}h(f_{0,\infty}^n,\mathscr{A})+\frac{\epsilon}{n}
\leq\frac{1}{n}h\big(f_{0,\infty}^n,{\mathscr{A}}_{0}^{n}(f_{0,\infty})\big)+\frac{\epsilon}{n}\\
&\leq\frac{1}{n}h(f_{0,\infty}^n,\mathscr{C}_n)+\frac{\epsilon}{n}
\leq\frac{1}{n}\log{|\mathscr{B}_n|}+\frac{\epsilon}{n}
=\frac{1}{n}\log\mathcal{N}({\mathscr{A}}_{0}^{n}(f_{0,\infty})|_{\Omega(f_{0,\infty})})+\frac{\epsilon}{n}.
\end{align*}
Therefore,
$h(f_{0,\infty})\leq h\big(f_{0,\infty},\Omega(f_{0,\infty})\big)$.
\end{proof}

\begin{lemma}\cite{Kolyada96}\label{131}
Let $f_{0,\infty}$ be a sequence of continuous self-maps of a compact topological space $X$. Then for every $1\leq i\leq j$
and every open cover $\mathscr{A}$ of $X$, $h(f_{i,\infty},\mathscr{A})\leq h(f_{j,\infty},\mathscr{A})$ and
$h(f_{i,\infty})\leq h(f_{j,\infty})$.
\end{lemma}

Based on Lemma \ref{131}, the concept of asymptotical topological entropy is introduced for $f_{0,\infty}$ on a compact
topological space in \cite{Kolyada96}. Define
\[h^{*}(f_{\infty})=\sup\{h^{*}(f_{\infty},\mathscr{A}): \mathscr{A}\; {\rm is\; an\; open\; cover\; of\;} X\;\}\]
as the asymptotical topological entropy of  $(X,f_{0,\infty})$, where
\[h^{*}(f_{\infty},\mathscr{A})=\lim_{n\to\infty}h(f_{n,\infty},\mathscr{A}).\]
It is easy to see that
\[h^{*}(f_{\infty})=\lim_{n\to\infty}h(f_{n,\infty}).\]

Recall that a sequence of continuous self-maps $\{f_n\}_{n=0}^{\infty}$ on a compact uniform space $(X,\mathcal{U})$ converges uniformly to $f$ if for any $\gamma\in\mathcal{U}$, there exists $N\geq1$ such that $(f_n(x),f(x))\in\gamma$ for any $x\in X$ and $n\geq N$.

\begin{theorem}\label{s}
Let $f_{0,\infty}$ be a sequence of continuous self-maps on a compact uniform space $(X,\mathcal{U})$.
If $f_{0,\infty}$ converges uniformly to $f$, then
\[h(f_{0,\infty})\leq \cdots\leq h(f_{n,\infty})\leq\cdots\leq h^{*}(f_{\infty})\leq h(f),\;n\geq1.\]
\end{theorem}

\begin{proof}
Let $\{x_n\}_{n=0}^{\infty}\subset[0,1]$ be a sequence of mutually different points converging to a point $z$.
Define a continuous self-map $F$ on $(\{x_n\}_{n=0}^{\infty}\cup\{z\})\times X$ by
\[F(z,y)=(z,f(y)),\;y\in X,\]
and
\[F(x_n,y)=(x_{n+1},f_n(y)),\;y\in X,\;n\geq0.\]
It is easy to see that $\Omega(F)\subset\{z\}\times X$. Then
\[h(F)=h(F,\Omega(F))=h(F,\{z\}\times X).\]
Thus,
\begin{align}\label{113}
h(F,\{x_n\}\times X)\leq h(F,\{z\}\times X),\;n\geq0.
\end{align}
It follows from the definitions of topological entropy that
\[h(F,\{x_n\}\times X)=h(f_{n,\infty}),\;n\geq0,\]
and
\[h(F,\{z\}\times X)=h(f).\]
This, together with (\ref{113}), implies that
\[h(f_{n,\infty})\leq h(f),\;n\geq0.\]
Hence,
\[h^{*}(f_{\infty})=\lim_{n\to\infty}h(f_{n,\infty})\leq h(f).\]
\end{proof}

\begin{remark}
Theorems \ref{h} and \ref{s} are generalizations of Theorems H and E in \cite{Kolyada96} to a uniform space.
\end{remark}

Let $\mathscr{K}(X)$ be the hyperspace of all nonempty compact subsets of $X$.
Recall that a Borel probability measure $\mu$ on $X$ is $f_{0,\infty}$-homogeneous, if

{\rm (i)} $\mu(K)<\infty$ for all $K\in\mathscr{K}(X)$;

{\rm (ii)} $\mu(K)>0$ for some $K\in\mathscr{K}(X)$;

{\rm (iii)} for any $\gamma\in{\mathcal{U}}^{s,o}$, there exist $\eta\in{\mathcal{U}}^{s,o}$ and $c>0$ such that
for any $x,y\in X$ and $n\geq1$,
\begin{align}\label{ying}
\mu(D_n(f_{0,\infty},y,\eta))\leq c\mu(D_n(f_{0,\infty},x,\gamma)),
\end{align}
where
\[D_n(f_{0,\infty},x,\gamma)=\bigcap_{k=0}^{n-1}f_{0}^{-k}(\gamma[f_{0}^{k}(x)]).\]
For such a $\mu$, define
\begin{align}\label{ooo}
k(f_{0,\infty},\mu)=\lim_{\eta\in{\mathcal{U}}^{s,o}}\limsup_{n\to\infty}-\frac{1}{n}\log \mu\big(D_n(f_{0,\infty},y,\eta)\big).
\end{align}
Note that $k(f_{0,\infty},\mu)$ does not depend on $y$ used by (\ref{ying}).
Then we have the following calculation of topological entropy.

\begin{theorem}\label{h1l}
Let $f_{0,\infty}$ be a sequence of continuous self-maps on a compact uniform space $(X,\mathcal{U})$.
If there exists a Borel probability measure $\mu$ on $X$ which is $f_{0,\infty}$-homogeneous, then
\[h(f_{0,\infty})=k(f_{0,\infty},\mu),\]
where $k(f_{0,\infty},\mu)$ is specified in {\rm(\ref{ooo})}.
\end{theorem}

\begin{proof}
Let $\gamma\in{\mathcal{U}}^{s,o}$, $n\geq1$ and $E\subset X$ be an $(n,\gamma)$-separated set under $f_{0,\infty}$ with $|E|=s_n(f_{0,\infty},X,\gamma)$. Clearly, there exists $\eta\in{\mathcal{U}}^{s,o}$ such that $\eta^2\subset\gamma$.
It is easy to verify that
\[D_n(f_{0,\infty},x_1,\eta)\cap D_n(f_{0,\infty},x_2,\eta)=\emptyset \;{\rm if}\;x_1\neq x_2\in E.\]
Since $\mu$ is $f_{0,\infty}$-homogeneous, there exists $\gamma'\in{\mathcal{U}}^{s,o}$ and $c>0$ such that for all $x,y\in X$,
\[\mu(D_n(f_{0,\infty},y,\gamma'))\leq c\mu(D_n(f_{0,\infty},x,\eta)).\]
Then
\begin{align*}
s_n(f_{0,\infty},X,\gamma)\mu(D_n(f_{0,\infty},y,\gamma'))&\leq c\sum_{x\in E}\mu(D_n(f_{0,\infty},x,\eta))\\
&=c\mu\big(\bigcup_{x\in E}D_n(f_{0,\infty},x,\eta)\big)\leq c\mu(X).
\end{align*}
By the fact that $\mu(X)<\infty$, one gets that
\[\limsup_{n\to\infty}\frac{1}{n}\log s_n(f_{0,\infty},X,\gamma)\leq\limsup_{n\to\infty}-\frac{1}{n}\log\mu(D_n(f_{0,\infty},y,\gamma')).\]
Then
\[\lim_{\gamma\in{\mathcal{U}}^{s,o}}\limsup_{n\to\infty}\frac{1}{n}\log s_n(f_{0,\infty},X,\gamma)\leq\lim_{\gamma\in{\mathcal{U}}^{s,o}}\limsup_{n\to\infty}-\frac{1}{n}\log\mu(D_n(f_{0,\infty},y,\gamma')),\]
which implies that
$h(f_{0,\infty})\leq k(\mu,f_{0,\infty})$.

On the other hand, let $\gamma\in{\mathcal{U}}^{s,o}$ and $K\in\mathscr{K}(X)$ be the set satisfying that $\mu(K)>0$. Then there exist $\eta\in{\mathcal{U}}^{s,o}$ and $c>0$ such that for all $n\geq1$ and $x,y\in X$,
\[\mu(D_n(f_{0,\infty},x,\eta))\leq c\mu(D_n(f_{0,\infty},y,\gamma)).\]
Let $n\geq1$ and $F$ be a subset of $K$ which $(n,\eta)$-spans $K$ with $|F|=r_n(f_{0,\infty},K,\eta)$. Then
\[K\subset\bigcup_{x\in F}D_n(f_{0,\infty},x,\eta).\]
Thus,
\begin{align*}
0<\mu(K)\leq\mu(\bigcup_{x\in F}D_n(f_{0,\infty},x,\eta))
\leq\sum_{x\in F}\mu(D_n(f_{0,\infty},x,\eta))
\leq cr_n(f_{0,\infty},K,\eta)\mu(D_n(f_{0,\infty},y,\gamma)),
\end{align*}
which implies that
\[\limsup_{n\to\infty}\frac{1}{n}\log r_n(f_{0,\infty},K,\eta)\geq\limsup_{n\to\infty}-\frac{1}{n}\log\mu(D_n(f_{0,\infty},y,\gamma)).\]
So,
\[\lim_{\gamma\in{\mathcal{U}}^{s,o}}\limsup_{n\to\infty}\frac{1}{n}\log r_n(f_{0,\infty},\eta,K)\geq\lim_{\gamma\in{\mathcal{U}}^{s,o}}\limsup_{n\to\infty}-\frac{1}{n}\log\mu(D_n(f_{0,\infty},y,\gamma)).\]
Hence,
\[h(f_{0,\infty})\geq h(f_{0,\infty},K)\geq k(\mu,f_{0,\infty}).\]
\end{proof}

\begin{remark}
Theorem \ref{h1l} extends Proposition 7 in \cite{Bowen71} to a nonautonomous dynamical system generated by a sequence of continuous seif-maps on a compact uniform space.
\end{remark}

\section{Topological equi-semi-conjugacy}

In this section, the relations of topological entropy between two topologically equi-semicon- jugate nonautonomous
dynamical systems are studied.

The sequence $\{\Lambda_n\}_{n=0}^{\infty}$ of subsets of $X$ is called invariant under $(X,f_{0,\infty})$
if $f_n(\Lambda_n)\subset\Lambda_{n+1}$ for all $n\geq0$. Then $(X,f_{0,\infty})$ restricted to $\{\Lambda_n\}_{n=0}^{\infty}$
is called the invariant subsystem of $(X,f_{0,\infty})$ on $\{\Lambda_n\}_{n=0}^{\infty}$.

\begin{definition}
Let $\{\Lambda_n\}_{n=0}^{\infty}$ and $\{D_n\}_{n=0}^{\infty}$ be invariant under $(X,f_{0,\infty})$ and $(Y,g_{0,\infty})$,
respectively. If, for each $n\geq0$, there exists a surjective map $h_n:\Lambda_n\to D_n$ such that $h_{n+1}\circ f_n=g_n\circ h_n$,
and the sequence of maps $\{h_n\}_{n=0}^{\infty}$ is equi-continuous, then the invariant subsystem of $(X,f_{0,\infty})$ on $\{\Lambda_n\}_{n=0}^{\infty}$ is said to be topologically $\{h_n\}_{n=0}^{\infty}$-equi-semiconjugate to the invariant subsystem
of $(Y,g_{0,\infty})$ on $\{D_n\}_{n=0}^{\infty}$. Furthermore, if $\{h_n^{-1}\}_{n=0}^{\infty}$ is also equi-continuous,
they are said to be topologically $\{h_n\}_{n=0}^{\infty}$-equi-conjugate.
In the case that $\Lambda_n=X$ and $D_n=Y$ for each $n\geq0$, $(X,f_{0,\infty})$ is said to be topologically $\{h_n\}_{n=0}^{\infty}$-equi-(semi)conjugate to $(Y,g_{0,\infty})$.
If there exists $c>0$ such that $\sup_{y\in Y}|h_{n}^{-1}(y)|\leq c$ for any $n\geq0$, then $\{h_n\}_{n=0}^{\infty}$
is called finite-to-one.
\end{definition}

Recall that $\{h_n\}_{n=0}^{\infty}$ is equi-continuous if for any $\gamma\in{\mathcal{V}}^{s,o}$, there exists $\eta\in{\mathcal{U}}^{s,o}$ such that $(h_n(x),h_n(y))\in \gamma$ for each $n\geq0$ and any $x,y\in\Lambda_n$ with $(x,y)\in\eta$.

\begin{remark}
It is well known that if two autonomous dynamical systems are topologically conjugate, then their topological properties are all the same. However, this is not necessarily true for nonautonomous dynamical systems, even if two systems are topologically equi-conjugate,
see example 3.1 in \cite{Shi09}. Fortunately, topological entropy is an invariant under topological equi-conjugacy.
\end{remark}

\begin{theorem}\label{277}
Let $f_{0,\infty}$ and $g_{0,\infty}$ be two sequences of continuous self-maps on compact uniform spaces $(X,\mathcal{U})$
and $(Y,\mathcal{V})$, respectively, $\Lambda_n\subset X$ and $D_n\subset Y$ for each $n\geq0$. Assume that there exists
a sequence of maps $h_n: \Lambda_n\to D_n$, $n\geq0$, such that the invariant subsystem of $(X,f_{0,\infty})$ on
$\{\Lambda_n\}_{n=0}^{\infty}$ is topologically $\{h_n\}_{n=0}^{\infty}$-equi-semiconjugate to the invariant subsystem
of $(Y,g_{0,\infty})$ on $\{D_n\}_{n=0}^{\infty}$. Then
\[h(f_{0,\infty},\Lambda_0)\geq h(g_{0,\infty},D_0).\]
Consequently, if the invariant subsystem of $(X,f_{0,\infty})$ on $\{\Lambda_n\}_{n=0}^{\infty}$ is topologically $\{h_n\}_{n=0}^{\infty}$-equi-conjugate to the invariant subsystem of $(Y,g_{0,\infty})$ on $\{D_n\}_{n=0}^{\infty}$,  then
\[h(f_{0,\infty},\Lambda_0)=h(g_{0,\infty},D_0).\]
\end{theorem}

\begin{proof}
Since $\{h_n\}_{n=0}^{\infty}$ is equi-continuous, for any $\gamma\in{\mathcal{V}}^{s,o}$, there exists $\eta\in{\mathcal{U}}^{s,o}$
such that for any $n\geq0$ and $x,y\in\Lambda_n$,
\begin{align}\label{18}
(x,y)\in \eta\Rightarrow(h_n(x),h_n(y))\in \gamma.
\end{align}
Let $n\geq0$ and $E\subset D_0$ be an $(n,\gamma)$-separated set under $g_{0,\infty}$ with $|E|=s_n(g_{0,\infty},D_0,\gamma)$.
Fix $x_e\in h_{0}^{-1}(e)\subset\Lambda_0$ for any $e\in E$ and denote $E'=\{x_e: e\in E\}$. Then $|E'|=|E|$.
We claim that $E'$ is an $(n,\eta)$-separated set of $\Lambda_0$ under $f_{0,\infty}$.
In fact, for any $x_{e_1}\neq x_{e_2}\in E'$, $e_1\neq e_2\in E$, and then there exists $0\leq k\leq n-1$ such that $(g_{0}^{k}(e_1),g_{0}^{k}(e_2))\notin\gamma$. Since
\[g_{0}^{k}(e_i)=g_{0}^{k}\circ h_0(x_{e_i})=h_k\circ f_{0}^{k}(x_{e_i}),\;i=1,2,\]
$(h_k\circ f_{0}^{k}(x_{e_1}),h_k\circ f_{0}^{k}(x_{e_2}))\notin \gamma$. It follows from (\ref{18}) that
$(f_{0}^{k}(x_{e_1}),f_{0}^{k}(x_{e_2}))\notin\eta$. Thus,
\[s_n(g_{0,\infty},D_0,\gamma)\leq s_n(f_{0,\infty},\Lambda_0,\eta),\;n\geq0,\;\gamma\in{\mathcal{V}}^{s,o},\]
which implies that
\begin{align*}
\lim_{\gamma\in{\mathcal{V}}^{s,o}}\limsup_{n\to\infty}\frac{1}{n}\log s_n(g_{0,\infty},D_0,\gamma)
&\leq\lim_{\gamma\in{\mathcal{V}}^{s,o}}\limsup_{n\to\infty}\frac{1}{n}\log s_n(f_{0,\infty},\Lambda_0,\eta)\\
&\leq\lim_{\eta\in{\mathcal{U}}^{s,o}}\limsup_{n\to\infty}\frac{1}{n}\log s_n(f_{0,\infty},\Lambda_0,\eta).
\end{align*}
Hence, $h(f_{0,\infty},\Lambda_0)\geq h(g_{0,\infty},D_0)$.
\end{proof}

In the special case that $\Lambda_n=X$ and $D_n=Y$ for all $n\geq0$, we get the following result.

\begin{Corollary}\label{515}
Let $f_{0,\infty}$ and $g_{0,\infty}$ be two sequences of continuous self-maps on compact uniform spaces $(X,\mathcal{U})$
and $(Y,\mathcal{V})$, respectively. If $(X,f_{0,\infty})$ is topologically equi-semiconjugate to $(Y,g_{0,\infty})$, then
\[h(f_{0,\infty})\geq h(g_{0,\infty}).\]
Consequently, if $(X,f_{0,\infty})$ is topologically equi-conjugate to $(Y,g_{0,\infty})$, then
\[h(f_{0,\infty})=h(g_{0,\infty}).\]
\end{Corollary}

Let $R$ be an equivalent relation on $X$ satisfying that for any $\gamma\in\mathcal{U}$, there exists $\eta\in\mathcal{U}$ such that
$\eta\circ R\circ\eta\subset R\circ\gamma\circ R$. Then the quotient map $\xi:X\to X/R$ generates a uniformity $\mathcal{U}/R$ on $X/R$ \cite{Cech66,Hood74}, where
\[\mathcal{U}/R=\{\gamma'\subset X/R\times X/R: (\xi\times \xi)^{-1}\gamma'\in\mathcal{U}\}=\{\gamma_R=(\xi\times \xi)\gamma: \gamma\in\mathcal{U}\}.\]
Note that
\begin{align}\label{009}
(\xi\times \xi)^{-1}\gamma_R=(\xi\times \xi)^{-1}(\xi\times \xi)\gamma=R\circ\gamma\circ R.
\end{align}
Let $\mathcal{U}^R$ be the uniformity on $X$ with base $\{R\circ\gamma\circ R: \gamma\in\mathcal{U}\}$.

\begin{lemma}\cite{Hood74}\label{31}
Let $f$ be a continuous self-map on a compact uniform spaces $(X,\mathcal{U})$ and $s$ be a self-map on $X/R$ satisfying that
$s\circ\xi=\xi\circ f$. Then $f$ is continuous on $(X,\mathcal{U}^R)$ and $s$ is continuous on $(X/R,\mathcal{U}/R)$.
\end{lemma}

\begin{theorem}\label{32}
Let $f_{0,\infty}$ be a sequence of continuous self-maps on a compact uniform space $(X,\mathcal{U})$.
Assume that there exists a sequence of maps $s_{0,\infty}=\{s_n\}_{n=0}^{\infty}$ on $(X/R,\mathcal{U}/R)$
satisfying that $s_n\circ\xi=\xi\circ f_n$ for any $n\geq0$. Then
\[h(f_{0,\infty},\mathcal{U}^{R})\leq h(s_{0,\infty},\mathcal{U}/R)\leq h(f_{0,\infty},\mathcal{U}).\]
\end{theorem}

\begin{proof}
For any $n\geq0$, $s_n(X/R)\subset X/R$ since $\xi$ is surjective and $s_n\circ\xi=\xi\circ f_n$, and $s_n$ is continuous on $(X/R,\mathcal{U}/R)$ by Lemma \ref{31}. By the fact that $\xi$ is continuous and surjective, $(X,\mathcal{U},f_{0,\infty})$ is topologically equi-semiconjugate to $(X/R,\mathcal{U}/R,s_{0,\infty})$. Thus, $h(s_{0,\infty},\mathcal{U}/R)\leq h(f_{0,\infty},\mathcal{U})$
by Theorem \ref{277}.

Let $\tilde{F}$ be a subset of $X/R$ which $(n,\gamma_R)$-spans $X/R$ under $s_{0,\infty}$ with $|\tilde{F}|=r_n(\gamma_R,X/R,s_{0,\infty})$.
For any $x\in X$, there exists $\tilde{y}\in\tilde{F}$ such that $(s_{0}^{j}\circ\xi(x),s_{0}^{j}(\tilde{y}))\in\gamma_R$ for any
$0\leq j\leq n-1$. For any $\tilde{z}\in \tilde{F}$, take one point $z\in X$ such that $\xi(z)=\tilde{z}$. Denote
$F=\{z: \tilde{z}\in F\}$. Then $|F|=|\tilde{F}|$ and $\xi(F)=\tilde{F}$. Thus, there exists $y\in F$ such that $\tilde{y}=\xi(y)$. So,
\[(\xi\circ f_{0}^{j}(x),\xi\circ f_{0}^{j}(y))=(s_{0}^{j}\circ\xi(x),s_{0}^{j}\circ\xi(y))
=(s_{0}^{j}\circ\xi(x),s_{0}^{j}(\tilde{y}))\in\gamma_R,\;0\leq j\leq n-1.\]
This, together with (\ref{009}), implies that
\[(f_{0}^{j}(x),f_{0}^{j}(y))\in(\xi\times\xi)^{-1}\gamma_R=R\circ\gamma\circ R,\;0\leq j\leq n-1.\]
Hence, $F$ is a subset of $X$ which $(n,R\circ\gamma\circ R)$-spans $X$ under $f_{0,\infty}$ and
\[r_n(f_{0,\infty},X,R\circ\gamma\circ R)\leq|F|=|\tilde{F}|=r_n(s_{0,\infty},X/R,\gamma_R).\]
Therefore,
$h(s_{0,\infty},\mathcal{U}/R)\geq h(f_{0,\infty},\mathcal{U}^{R})$.
\end{proof}

\begin{remark}
There may be many different uniformities on the set $X$. Theorem \ref{32} shows that topological entropy depends on the
uniformity on $X$.
\end{remark}

The topological sup-entropy of  $f_{0,\infty}$ on a subset of $X$ is proposed in \cite{Kolyada96}, we extend it to uniform spaces.
Suppose that $f_{0,\infty}$ is equi-continuous on $X$. Let $n\geq1$ and $\gamma\in{\mathcal{U}}^{s,o}$.  A subset $E^*\subset X$ is called $(n,\gamma)^*$-separated if for each pair $x\neq y\in E^*$, there exist $i\geq0$ and $0\leq j\leq n-1$ such that $(f_i^j(x),f_i^j(y))\notin\gamma$; a subset $F^*\subset X$ is called $(n,\gamma)^*$-spans another set $K\subset X$ if for each $x\in K$ there exists $y\in F^*$ such that $(f_i^j(x),f_i^j(y))\in\gamma$ for any $0\leq j\leq n-1$ and $i\geq0$. For a nonempty subset $\Lambda$ of $X$, let $s_n^*(f_{0,\infty},\Lambda,\gamma)$ be the maximal cardinality of an $(n,\gamma)^*$-separated set in $\Lambda$ under $f_{0,\infty}$ and $r_n^*(f_{0,\infty},\Lambda,\gamma)$ be the minimal cardinality of a set in $\Lambda$ which $(n,\gamma)^*$-spans $\Lambda$ under $f_{0,\infty}$. $s_n^*(f_{0,\infty},\Lambda,\gamma)$ and $r_n^*(f_{0,\infty},\Lambda,\gamma)$ are finite since $X$ is compact and $f_{0,\infty}$ is equi-continuous. With a similar proof to that of Proposition \ref{1}, one easily shows that for any $\gamma_1,\gamma_2\in\mathcal{U}^{s,o}$ with $\gamma_1\subset\gamma_2$,
\[s_n^*(f_{0,\infty},\Lambda,\gamma_1)\geq s_n^*(f_{0,\infty},\Lambda,\gamma_2),\;r_n^*(f_{0,\infty},\Lambda,\gamma_1)\geq r_n^*(f_{0,\infty},\Lambda,\gamma_2),\]
and for any $\gamma,\eta\in{\mathcal{U}}^{s,o}$ with $\eta^2\subset\gamma$,
\[r_n^*(f_{0,\infty},\Lambda,\gamma)\leq s_n^*(f_{0,\infty},\Lambda,\gamma)\leq r_n^*(f_{0,\infty},\Lambda,\eta).\]
Define the topological sup-entropy of  $f_{0,\infty}$ on the subset $\Lambda$ of $X$ as
\[H(f_{0,\infty},\Lambda)=\lim_{\gamma\in{\mathcal{U}}^{s,o}}\limsup_{n\to\infty}\frac{1}{n}\log s_n^*(f_{0,\infty},\Lambda,\gamma)
=\lim_{\gamma\in{\mathcal{U}}^{s,o}}\limsup_{n\to\infty}\frac{1}{n}\log r_n^*(f_{0,\infty},\Lambda,\gamma).\]
Clearly, $h(f_{0,\infty},\Lambda)\leq H(f_{0,\infty},\Lambda)$.

\begin{theorem}\label{h1}
Let $f_{0,\infty}$ be a sequence of equi-continuous self-maps on a compact Hausdorff uniform space $(X,\mathcal{U})$ and $g_{0,\infty}$ be a sequence of continuous self-maps on a compact Hausdorff uniform space $(Y,\mathcal{V})$. Assume that $(X,f_{0,\infty})$ is topologically $\{\pi_n\}_{n=0}^{\infty}$-equi-semiconjugate to $(Y,g_{0,\infty})$ and $\{\pi_n\}_{n=0}^{\infty}\subset\{\phi_1,\cdots,\phi_k\}$
for some $k\geq1$. Then
\[h(g_{0,\infty})\leq h(f_{0,\infty})\leq h(g_{0,\infty})+\max_{1\leq j\leq k}\sup_{y\in Y}H(f_{0,\infty},\phi_{j}^{-1}(y)).\]
Consequently, if $\{\pi_n\}_{n=0}^{\infty}$ is finite-to-one, then
\[h(f_{0,\infty})=h(g_{0,\infty}).\]
\end{theorem}

\begin{proof}
It follows from Corollary \ref{515} that $h(f_{0,\infty})\geq h(g_{0,\infty})$.

Denote $a:=\max_{1\leq j\leq k}\sup_{y\in Y}H(f_{0,\infty},\phi_{j}^{-1}(y))$. It is to show that $h(f_{0,\infty})\leq h(g_{0,\infty})+a$.
If $a=\infty$, then we are done. Suppose that $a<\infty$.
Let $u\in{\mathcal{U}}^{s,o}$. Then there exists $\gamma\in{\mathcal{U}}^{s,o}$ such that $\gamma^{2}\subset u$.
By the definition of $H(f_{0,\infty},\phi_{j}^{-1}(y))$,
there exists $m_j(y)\geq1$ such that
\[\frac{1}{m_j(y)}\log r_{m_j(y)}^{*}(f_{0,\infty},\phi_{j}^{-1}(y),\gamma)\leq H(f_{0,\infty},\phi_{j}^{-1}(y))\leq a\]
for any $y\in Y$ and $1\leq j\leq k$, and let
$F^{*}_y(j)\subset\phi_{j}^{-1}(y)$ which $(m_j(y),\gamma)$-spans $\phi_{j}^{-1}(y)$ with $|F^{*}_y(j)|=r_{m_j(y)}^{*}(f_{0,\infty},\phi_{j}^{-1}(y),\gamma)$, then
\begin{align}\label{lo}
\frac{1}{m_j(y)}\log|F^{*}_y(j)|\leq a.
\end{align}
For any $y\in Y$ and $1\leq j\leq k$, denote
\begin{align}\label{lo2}
U_{y}^{j}=\bigcup_{z\in F^{*}_y(j)}D^{*}_{m_j(y)}(f_{0,\infty},z,\gamma),
\end{align}
where
\begin{align}\label{lo3}
D^{*}_{m_j(y)}(f_{0,\infty},z,\gamma)=\{\omega\in X: (f_{i}^{t}(w),f_{i}^{t}(z))\in\gamma, \;0\leq t\leq m_j(y)-1,\;
i\geq0\}.
\end{align}
Since $f_{0,\infty}$ is equi-continuous, $D^{*}_{m_j(y)}(f_{0,\infty},z,\gamma)$ is open, and thus
$U_{y}^{j}$ is open. Clearly, $\phi_{j}^{-1}(y)\subset U_{y}^{j}$, which implies that
\[(X\setminus U_{y}^{j})\bigcap(\bigcap_{\gamma\in{\mathcal{U}}^{s,o}}\phi_{j}^{-1}(\bar{\gamma}[y]))=\emptyset.\]
Thus, there exists $\gamma_{j,y}\in{\mathcal{U}}^{s,o}$ such that $\phi_{j}^{-1}(\gamma_{j,y}[y])\subset U_{y}^{j}$ since $X$ is compact.
Denote
$W_y=\bigcap_{1\leq j\leq k}\gamma_{j,y}[y].$
Then
\begin{align}\label{ll}
\phi_{j}^{-1}(W_{y})\subset U_{y}^{j},\;1\leq j\leq k.
\end{align}

Let $\{W_{y_1},\cdots,W_{y_p}\}$, $p\geq1$, be a subcover of the open cover $\{W_{y}: y\in X\}$.
By Lemma \ref{Lebesgue covering lemma}, there exists $\eta\in{\mathcal{U}}^{s,o}$ such that $\eta[x]$ is a subset of some member of $\{W_{y_1},\cdots,W_{y_p}\}$ for any $x\in X$. For each $n\geq0$, there exists $1\leq j_n\leq k$ such that
\begin{align}\label{5151}
\pi_n=\phi_{j_n}.
\end{align}
Let $n\geq1$ and $E_n$ be a subset of $Y$ which $(n,\eta)$-spans $Y$ under $g_{0,\infty}$ with $|E_n|=r_n(g_{0,\infty},Y,\eta)$.
For any $y\in E_n$, there exists $c_0(y)\in\{y_1,\cdots,y_p\}$ such that $\eta[y]\subset W_{c_0(y)}$, and define $t_0=0$;
let $t_1(y)=m_{j_0}(c_0(y))$, and let $c_1(y)\in\{y_1,\cdots,y_p\}$ such that $\eta[g_0^{t_1(y)}(y)]\subset W_{c_1(y)}$.
Inductively, assume that $t_0(y),\cdots,t_k(y)$ and $c_0(y),\cdots,c_k(y)$ are already defined, define
\[t_{k+1}(y)=t_k(y)+m_{j_{t_{k}}}(c_k(y)),\]
and define $c_{k+1}(y)\in\{y_1,\cdots,y_p\}$ satisfying that
\begin{align}\label{lll}
\eta[g_0^{t_{k+1}(y)}(y)]\subset W_{c_{k+1}(y)}.
\end{align}
In fact,
\begin{align}\label{234}
t_{k}(y)=\sum_{s=0}^{k-1}m_{j_{t_{s}}}(c_s(y)),\; k\geq1.
\end{align}
Then there exists $l\geq0$ such that
\begin{align}\label{345}
t_l(y)<n\leq t_{l+1}(y).
\end{align}

We claim that
\[\beta=\{V(y;x_0,\cdots,x_l): y\in E_n, x_s\in F^{*}_{c_s(y)}(j_{t_{s}(y)}),\;0\leq s\leq l\}\]
is an open cover of $X$, where
\[V(y;x_0,\cdots,x_l)=\{x\in X: (f_{0}^{t+t_s(y)}(x),f_{t_s(y)}^{t}(x_s))\in\gamma,\;0\leq t\leq m_{j_{t_{s}(y)}}(c_s(y))-1, \;0\leq s\leq l\}.\]
In fact, let $x\in X$. Since $E_n$ is a set of $Y$ which $(n,\eta)$-spans $Y$ under $g_{0,\infty}$, there exists $y\in E_n$ such that
\begin{align}\label{lo1}
(g_{0}^{i}(y),g_{0}^{i}(\pi_0(x)))\in\eta,\; 0\leq i\leq n-1.
\end{align}
For any $0\leq s\leq l$,
\[t_{s}(y)\leq t_{l}(y)<n,\]
then by (\ref{lo1}), one has that
\[\pi_{t_{s}(y)}\circ f_{0}^{t_{s}(y)}(x)=g_{0}^{t_{s}(y)}\circ\pi_{0}(x)\in\eta[g_{0}^{t_{s}(y)}(y)],\]
and thus by (\ref{ll})-(\ref{lll}), we have
\[f_{0}^{t_{s}(y)}(x)\in\pi_{t_{s}(y)}^{-1}(\eta[g_{0}^{t_{s}(y)}(y)])=\phi_{j_{t_{s}(y)}}^{-1}(\eta[g_{0}^{t_{s}(y)}(y)])
\subset \phi_{j_{t_{s}(y)}}^{-1}(W_{c_{s}(y)})\subset U_{c_{s}(y)}^{j_{t_{s}(y)}}.\]
It follows from (\ref{lo2}) and (\ref{lo3}) that there exists $x_s\in F^{*}_{c_{s}(y)}(j_{t_{s}(y)})$ such that
\[(f_{i}^{t}(f_{0}^{t_{s}(y)}(x)),f_{i}^{t}(x_s))\in\gamma,\;0\leq t\leq m_{j_{t_{s}(y)}}(c_{s}(y))-1,\;i\geq0.\]
Thus,
\[(f_{0}^{t+t_{s}(y)}(x),f_{{t_{s}(y)}}^{t}(x_s))=(f_{{t_{s}(y)}}^{t}(f_{0}^{t_{s}(y)}(x)),f_{{t_{s}(y)}}^{t}(x_s))\in\gamma,
\;0\leq t\leq m_{j_{t_{s}(y)}}(c_{s}(y))-1,\;0\leq s\leq l.\]
So, $x\in V(y;x_0,\cdots,x_l)$. Hence, $\beta$ is an open cover of $X$.

We also claim that any $(n,u)$-separated set intersects each element of $\beta$ at most one point.
Otherwise, there exists an $(n,u)$-separated pair $(z,w)$ such that $z,w\in V(y;x_0,\cdots,x_l)$
for some $y\in E_n$ and $x_s\in F^{*}_{c_s(y)}(j_{t_{s}(y)})$, $0\leq s\leq l$. Then
\[(f_{0}^{t+t_s(y)}(z),f_{t_s(y)}^{t}(x_s))\in\gamma,\;(f_{0}^{t+t_s(y)}(\omega),f_{t_s(y)}^{t}(x_s))\in\gamma,\;0\leq t\leq m_{j_{t_{s}(y)}}(c_s(y))-1,\;0\leq s\leq l,\]
which yields that
\begin{align}\label{0091}
(f_{0}^{t+t_s(y)}(z),f_{0}^{t+t_s(y)}(\omega))\in \gamma\circ\gamma\subset u,\;0\leq t\leq m_{j_{t_{s}(y)}}(c_s(y))-1,\;0\leq s\leq l.
\end{align}
This contradicts to the fact that $(z,w)$ is an $(n,u)$-separated pair for $(X,f_{0,\infty})$ since
\[\max_{0\leq s\leq l}\big(t+t_s(y)\big)=t_{l+1}(y)-1\geq n-1.\]
So,
\begin{align}\label{0092}
s_n(f_{0,\infty},X,u)\leq|\beta|=|E_n|\Pi_{s=0}^{l}|F^{*}_{c_s(y)}(j_{t_{s}(y)})|.
\end{align}
Hence, by (\ref{lo}), (\ref{234}), (\ref{345}) and (\ref{0092}), we have that
\begin{align*}
\frac{1}{n}\log s_n(f_{0,\infty}, X,u)&\leq\frac{1}{n}\log|E_n|+\frac{1}{n}\sum_{s=0}^{l}\log|F^{*}_{c_s(y)}(j_{t_{s}(y)})|\\
&\leq\frac{1}{n}\log r_n(g_{0,\infty},Y,\eta)+\frac{a}{n}\sum_{s=0}^{l}\log m_{j_{t_{s}(y)}}(c_s(y))\\
&=\frac{1}{n}\log r_n(g_{0,\infty},Y,\eta)+\frac{a}{n}\big(\sum_{s=0}^{l-1}\log m_{j_{t_{s}(y)}}(c_s(y))+\log m_{j_{t_{l}(y)}}(c_l(y))\big)\\
&\leq\frac{1}{n}\log r_n(g_{0,\infty},Y,\eta)+\frac{a}{n}(n+M),
\end{align*}
where $M:=\max_{1\leq j\leq k}\{m_{j}(y_1),\cdots,m_{j}(y_p)\}$.
This implies that
\[h(f_{0,\infty})\leq h(g_{0,\infty})+a.\]

If $\{\pi_n\}_{n=0}^{\infty}$ is finite-to-one, then $a=0$, and thus
$h(f_{0,\infty})\leq h(g_{0,\infty}).$
\end{proof}

\begin{remark}
{\rm(i)} Theorem \ref{h1} is a generalization of Theorem 17 in \cite{Bowen71} to nonautonomous dynamical systems, and it is also a uniform version of Theorem C in \cite{Kolyada96}.

{\rm(ii)} For two semi-conjugate random dynamical systems on Polish spaces, Liu in \cite{liu05} proved that they have the same entropy
if the cardinal number of the pre-image of a point under the semi-conjugacy is finite almost everywhere.
\end{remark}

\section{Estimations of topological entropy for $A$-coupled-expanding systems}

In this section, some estimations of upper and lower bounds of topological entropy for an invariant subsystem of an
$A$-coupled-expanding system are obtained.

Let us recall the definitions of subshifts of finite type \cite{Rob99}. A matrix $A=(a_{ij})_{N\times N}$ $(N\geq2)$
is said to be a transition matrix if $a_{ij}=0$ or $1$ for all $i, j$; $\sum_{i=1}^{N}a_{ij}\geq1$ for all $j$;
and $\sum_{j=1}^{N}a_{ij}\geq1$ for all $i$, $1\leq i,j\leq N$.
Given a transition matrix $A=(a_{ij})_{N\times N}$, denote
\[\Sigma_{N}^{+}(A)=\{s=(s_0,s_1,\cdots): 1\leq s_j\leq N,\; a_{s_js_{j+1}}=1,\; j\geq0\}.\]
Note that $\Sigma_{N}^{+}(A)$ is a compact metric space with the metric
\[\rho(\alpha,\beta)=\sum_{i=0}^{\infty}\frac{d(a_i, b_i)}{2^i},\;\alpha=(a_0, a_1,\cdots),\;\beta=(b_0, b_1,\cdots)\in\Sigma_{N}^{+}(A),\]
where $d(a_i, b_i)=0$ if $a_i=b_i$, and $d(a_i, b_i)=1$ if $a_i\neq b_i$, $i\geq0$. The map
$\sigma_A: \Sigma_{N}^{+}(A)\to\Sigma_{N}^{+}(A)$ with
\[\sigma_A((s_0,s_1,s_2,\cdots))=(s_1,s_2,\cdots),\;(s_0,s_1,s_2\cdots)\in\Sigma_{N}^{+}(A),\]
is called a subshift of finite type associated with matrix $A$. Its topological entropy is equal to $\log\lambda(A)$, where
$\lambda(A)$ is the spectral radius of matrix $A$ and
\begin{align}\label{53}
\lambda(A)=\lim_{n\to\infty}{\parallel A^n\parallel}^{\frac{1}{n}},\;\parallel A\parallel=\sum_{1\leq i,j\leq N}a_{ij}.
\end{align}

\begin{definition}
Let $f_{0,\infty}$ be a sequence of self-maps on a uniform space $X$
and $A=(a_{ij})_{N\times N}$ be a transition matrix. If there exist $N$ nonempty subsets $V_1,\cdots,V_N$
of $X$ with pairwise disjoint interiors, such that
\[f_n(V_i)\supset\bigcup_{a_{ij}=1}V_j,\;1\leq i\leq N,\;n\geq0,\]
then $(X,f_{0,\infty})$ is called $A$-coupled-expanding in $V_i$, $1\leq i\leq N$.
Furthermore, $(X,f_{0,\infty})$ is said to be strictly $A$-coupled-expanding in $V_i$, $1\leq i\leq N$,
if $\bar{V}_i\cap\bar{V}_j=\emptyset$ for all $1\leq i\neq j\leq N$, where $\bar{V}_i$ denotes
the closure of the set $V_i$ with respect to $X$. In the special case that $a_{ij}=1$ for all
$1\leq i,j\leq N$, it is briefly called coupled-expanding or strictly coupled-expanding in
$V_i$, $1\leq i\leq N$.
\end{definition}

An estimation of lower bound of topological entropy for an invariant subsystem of $(X,f_{0,\infty})$ is given in the following result.

\begin{theorem}\label{hh21}
Let $f_{0,\infty}$ be a sequence of continuous self-maps on a compact Hausdorff uniform space $(X,\mathcal{U})$, $A=(a_{ij})_{N\times N}$ be a transition matrix and $V_1,\cdots,V_N$ be nonempty, closed and mutually disjoint subsets of $X$. Assume that
\begin{itemize}
\item[{\rm(i)}] $(X,f_{0,\infty})$ is $A$-coupled-expanding in $V_i$, $1\leq i\leq N$;

\item[{\rm(ii)}] $f_{0,\infty}$ is equi-continuous in $\bigcup_{i=1}^{N}V_i$.
\end{itemize}
Then, for each $n\geq0$, there exist a nonempty compact subset $\Lambda_n\subset\bigcup_{i=1}^{N}V_i$ with $f_n(\Lambda_n)=\Lambda_{n+1}$
and a map $h_n:\Lambda_n\to\Sigma_{N}^{+}(A)$ such that the invariant subsystem of $(X,f_{0,\infty})$ on $\{\Lambda_n\}_{n=0}^{\infty}$ is topologically $\{h_n\}_{n=0}^{\infty}$-equi-semiconjugate to $(\Sigma_{N}^{+}(A), \sigma_A)$.
Consequently,
\[h(f_{0,\infty},\Lambda_0)\geq\log\lambda(A),\]
where $\lambda(A)$ is specified in {\rm(\ref{53})}.
\end{theorem}

\begin{proof}
Since $(X,f_{0,\infty})$ is $A$-coupled-expanding in $V_i$, $1\leq i\leq N$, for any $m,n\geq0$ and  $\alpha=(a_0,a_1,\cdots)\in\Sigma_{N}^{+}(A)$,
\begin{align}\label{s1}
V_{\alpha}^{m,n}=\bigcap_{k=0}^{m}f_{n}^{-k}(V_{a_k})\neq\emptyset.
\end{align}
Fix $n\geq0$. Then $V_{\alpha}^{m,n}$ is a nonempty closed subset of $X$
and satisfies that $V_{\alpha}^{m+1,n}\subset V_{\alpha}^{m,n}$ for any $m\geq0$ and $\alpha\in\Sigma_{N}^{+}(A)$.
Thus, $\bigcap_{m=0}^{\infty}V_{\alpha}^{m,n}\neq\emptyset$ for all $\alpha\in\Sigma_{N}^{+}(A)$ by the compactness of $X$.
Denote
\begin{align}\label{s2}
\Lambda_n:=\bigcup_{\alpha\in\Sigma_{N}^{+}(A)}\bigcap_{m=0}^{\infty}V_{\alpha}^{m,n}.
\end{align}
Clearly, $\Lambda_n\neq\emptyset$ and $\Lambda_n\subset\bigcup_{i=1}^{N}V_i$.
It is easy to verify that $f_n(x)\in\bigcap_{m=0}^{\infty}V_{\sigma_{A}(\alpha)}^{m,n+1}$.
Thus, $f_n(\Lambda_n)\subset\Lambda_{n+1}$. One also easily verify that $f_n(\Lambda_n)\supset\Lambda_{n+1}$
by assumption (i) and the fact that $\sigma_{A}$ is surjective.
So, $f_n(\Lambda_n)=\Lambda_{n+1}$.

We claim that $\Lambda_n$ is a compact subset of $X$.
Let $\{z_l\}_{l=1}^{\infty}$ be a sequence which converges to a point $z\in X$.
Then, there exists $\alpha_l\in\Sigma_{N}^{+}(A)$ such that $z_l\in\bigcap_{m=0}^{\infty}V_{\alpha_l}^{m,n}$
for each $l\geq1$. Since $\Sigma_{N}^{+}(A)$ is compact, $\{\alpha_l\}_{l=1}^{\infty}$ has a convergent subsequence.
Without loss of generality, suppose that $\{\alpha_l\}_{l=1}^{\infty}$ converges to $\alpha=(a_0, a_1,\cdots)$.
So, for any $m\geq0$, there exists $k_m\geq1$ such that $V_{\alpha_l}^{m,n}=V_{\alpha}^{m,n}$ for all $l\geq k_m$.
Hence, $z_l\in V_{\alpha}^{m,n}$ for all $l\geq k_m$, which implies that $z\in V_{\alpha}^{m,n}$.
Therefore, $z\in\bigcap_{m=0}^{\infty}V_{\alpha}^{m,n}$, and consequently, $\Lambda_n$ is closed, and thus compact.

For any $x\in\Lambda_n$, there exists $\alpha\in\Sigma_{N}^{+}(A)$ such that $x\in\bigcap_{m=0}^{\infty}V_{\alpha}^{m,n}$.
Define $h_n(x)=\alpha$. Then the map $h_n:\Lambda_n\to\Sigma_{N}^{+}(A)$ is well defined since $V_i\cap V_j=\emptyset$
for all $1\leq i\neq j\leq N$. Clearly, $h_n$ is surjective. Moreover,
\[h_{n+1}\circ f_n(x)=\sigma_{A}(\alpha)=\sigma_{A}\circ h_n(x),\;x\in\Lambda_n.\]

Next, it is to show that $\{h_n\}_{n=0}^{\infty}$ is equi-continuous.
Since $V_1,\cdots,V_N$ are mutually disjoint closed subsets of compact Hausdorff space $X$,
there exists $\gamma\in{\mathcal{U}}^{s,o}$ such that for any $1\leq i\neq j\leq N$,
\begin{align}\label{s3}
\gamma\bigcap(V_i\times V_j)=\emptyset.
\end{align}
For any $\epsilon>0$, there exists $N\geq1$ such that $2^{-N}<\epsilon$.
By the equi-continuity of $f_{0,\infty}$ in $\bigcup_{i=1}^{N}V_i$,
there exists $\eta\in{\mathcal{U}}^{s,o}$ such that for any $n\geq0$
and any $x,y\in\Lambda_n$ with $h_n(x)=\alpha=(a_0, a_1,\cdots)$ and $h_n(y)=\beta=(b_0, b_1,\cdots)$,
\[(x,y)\in\eta\Rightarrow(f_{n}^{j}(x),f_{n}^{j}(y))\in\gamma,\;0\leq j\leq N.\]
It follows from (\ref{s1}) and (\ref{s2}) that
\[f_{n}^{j}(x)\in V_{a_j},\;f_{n}^{j}(y)\in V_{b_j},\;0\leq j\leq N.\]
This, together with (\ref{s3}), implies that
\[a_j=b_j,\;0\leq j\leq N.\]
Thus,
\[\rho(h_n(x),h_n(y))=\rho(\alpha,\beta)\leq2^{-N}<\epsilon.\]
Hence, $\{h_n\}_{n=0}^{\infty}$ is equi-continuous in $\{\Lambda_n\}_{n=0}^{\infty}$.
Therefore, the invariant subsystem of $(X,f_{0,\infty})$ on $\{\Lambda_n\}_{n=0}^{\infty}$
is topologically $\{h_n\}_{n=0}^{\infty}$-equi-semiconjugate to $(\Sigma_{N}^{+}(A),\sigma_{A})$,
and consequently
\[h(f_{0,\infty}, \Lambda_0)\geq h(\Sigma_{N}^{+}(A),\sigma_{A})=\log\lambda(A)\]
by Theorem \ref{277}.
\end{proof}

The next result gives an estimation of lower bound of topological entropy for the full system.

\begin{Corollary}\label{hh212}
Let all the assumptions in Theorem \ref{hh21} hold and $\bigcup_{i=1}^{N}V_i=X$ except that
assumption {\rm(i)} is replaced by
\begin{align}\label{00}
f_n(V_i)=\bigcup_{a_{ij}=1}V_j,\;1\leq i\leq N,\;n\geq0.
\end{align}
Then $(X,f_{0,\infty})$ is topologically equi-semiconjugate to $(\Sigma_{N}^{+}(A),\sigma_A)$. Consequently,
\[h(f_{0,\infty})\geq\log\lambda(A).\]
\end{Corollary}

\begin{proof}
By the proof of Theorem \ref{hh21}, it is only to show that $\Lambda_n=X$ for any $n\geq0$,
where $\Lambda_n$ is specified in (\ref{s2}). Let $n\geq0$. By the assumption that $\bigcup_{i=1}^{N}V_i=X$ and (\ref{00}),
there exists $\beta\in\Sigma_{N}^{+}(A)$ such that $x\in\bigcap_{m=0}^{\infty}V_{\beta}^{m,n}$ for any $x\in X$,
which implies that $X\subset\Lambda_n$. Hence, $\Lambda_n=X$. Then the conclusion follows from Theorem \ref{hh21}.
\end{proof}

To proceed, we need the following lemma.

\begin{lemma}\label{hh22}
Let $(X,\mathcal{U})$ and $(Y,\mathcal{V})$ be compact uniform spaces, $Y_n\subset Y$ and $\pi_n: X\to Y_n$ be a map
for each $n\geq0$. Assume that $\{\pi_n\}_{n=0}^{\infty}$ is equi-continuous at any point of $X$; that is, for any fixed point $x\in X$
and any $\gamma\in{\mathcal{V}}^{s,o}$, there exists $\eta\in{\mathcal{U}}^{s,o}$ such that $\pi_n(y)\in\gamma[\pi_n(x)]$ for any
$y\in\eta[x]$ and $n\geq0$. Then $\{\pi_n\}_{n=0}^{\infty}$ is equi-continuous in $X$.
\end{lemma}

\begin{proof}
Fix $\gamma'\in{\mathcal{V}}^{s,o}$. Then there exists $\gamma\in{\mathcal{V}}^{s,o}$ such that $\gamma^2\subset\gamma'$.
By the assumptions, for any $x\in X$, there exists $\eta_x\in{\mathcal{U}}^{s,o}$ such that for any $y\in X$,
\begin{align}\label{1011}
y\in\eta_x[x]\Rightarrow\pi_n(y)\in\gamma[\pi_n(x)],\;n\geq0.
\end{align}
For $\eta_x$, there exists $u_x\in{\mathcal{U}}^{s,o}$ such that $u_x^{2}\subset\eta_x$.
Since $\{u_x[x]: x\in X\}$ is an open cover of compact space $X$, there exists a finite subcover
$\{u_{x_{i}}[x_i]: 1\leq i\leq m\}$ for some $m\geq1$.
Denote $u=\bigcap_{i=1}^{m}u_{x_{i}}$.
Then $u\in{\mathcal{U}}^{s,o}$. So, for any $y_1,y_2\in X$ with $(y_1,y_2)\in u$, there exists $1\leq i_0\leq m$ such that
$y_1\in u_{x_{i_0}}[x_{i_0}]$. Then
\[(y_1,x_{i_0})\in u_{x_{i_0}}^2\subset\eta_{x_{i_0}},\;
(y_2,x_{i_0})\in u_{x_{i_0}}^2\subset\eta_{x_{i_0}}.\]
By (\ref{1011}), we have that
\begin{align}\label{1012}
(\pi_n(y_1),\pi_n(x_{i_0}))\in\gamma,\;(\pi_n(y_2),\pi_n(x_{i_0}))\in\gamma,\;n\geq0,
\end{align}
which yields that
\[(\pi_n(y_1),\pi_n(y_2))\in\gamma^2\subset\gamma',\;n\geq0.\]
Therefore, $\{\pi_n\}_{n=0}^{\infty}$ is equi-continuous in $X$.
\end{proof}

An estimation of upper bound of topological entropy for an invariant subsystem of $(X,f_{0,\infty})$ is given in the next result.

\begin{theorem}\label{hh2}
Let $f_{0,\infty}$ be a sequence of continuous self-maps on a compact Hausdorff uniform space $(X,\mathcal{U})$,
$A=(a_{ij})_{N\times N}$ be a transition matrix and $V_1,\cdots,V_N$ be nonempty and closed
subsets of $X$. Assume that
\begin{itemize}
\item[{\rm(i)}] $(X,f_{0,\infty})$ is $A$-coupled-expanding in $V_i$, $1\leq i\leq N$;

\item[{\rm(ii)}] for any $\alpha\in\Sigma_{N}^{+}(A)$ and $\gamma\in{\mathcal{U}}^{s,o}$,
there exists $M\geq1$ such that
\begin{align}\label{90}
m\geq M\Rightarrow V_{\alpha}^{m,n}\times V_{\alpha}^{m,n}\subset\gamma, \;n\geq0,
\end{align}
\end{itemize}
where $V_{\alpha}^{m,n}$ is specified in {\rm(\ref{s1})}.
Then, for each $n\geq0$, there exist a nonempty compact subset $\Lambda_n\subset\bigcup_{i=1}^{N}V_i$
with $f_n(\Lambda_n)=\Lambda_{n+1}$ and a map $\pi_n:\Sigma_{N}^{+}(A)\to\Lambda_n$ such that
$(\Sigma_{N}^{+}(A), \sigma_A)$ is topologically $\{\pi_n\}_{n=0}^{\infty}$-equi-semiconjugate
to the invariant subsystem of $(X,f_{0,\infty})$ on $\{\Lambda_n\}_{n=0}^{\infty}$. Consequently,
\[h(f_{0,\infty},\Lambda_0)\leq\log\lambda(A).\]
\end{theorem}

\begin{proof}
For each $n\geq0$, by assumption (i) and the continuity of $f_n$, it is easy to verify
that $V_{\alpha}^{m,n}$ is nonempty, closed and satisfies that $V_{\alpha}^{m+1,n}\subset V_{\alpha}^{m,n}$
for any $\alpha\in\Sigma_{N}^{+}(A)$ and $m\geq0$. This, together with assumption (ii) and
$\mathcal{U}$ is separated, yields that $\bigcap_{m=0}^{\infty}V_{\alpha}^{m,n}$ is a singleton set
for any $\alpha\in\Sigma_{N}^{+}(A)$ and $n\geq0$. Denote
\begin{align}\label{111}
\bigcap_{m=0}^{\infty}V_{\alpha}^{m,n}=\{x_n(\alpha)\},\;
\Lambda_n=\{x_n(\alpha):\alpha\in\Sigma_{N}^{+}(A)\}.
\end{align}
Clearly, $\Lambda_n\neq\emptyset$, $\Lambda_n\subset\bigcup_{i=1}^{N}V_i$, and
\[f_n(x_n(\alpha))=x_{n+1}(\sigma_{A}(\alpha)),\;\alpha\in\Sigma_{N}^{+}(A),\;n\geq0.\]
Hence, it follows from the fact that $\sigma_{A}$ is surjective that $f_n(\Lambda_n)=\Lambda_{n+1}$.

Define a map $\pi_n:\Sigma_{N}^{+}(A)\to\Lambda_n$ by $\pi_n(\alpha)=x_n(\alpha)$ for any $\alpha\in\Sigma_{N}^{+}(A)$ and $n\geq0$.
Clearly, $\pi_n$ is well defined and surjective. Moreover, we have
\[f_n\circ\pi_n(\alpha)=f_n(x_n(\alpha))=x_{n+1}(\sigma_{A}(\alpha))=\pi_{n+1}\circ\sigma_{A}(\alpha),\;\alpha\in\Sigma_{N}^{+}(A),\;n\geq0.\]
Next, it is to show that $\{\pi_n\}_{n=0}^{\infty}$ is equi-continuous in $\Sigma_{N}^{+}(A)$.
Fix any $\alpha=(a_0,a_1,\cdots)\in\Sigma_{N}^{+}(A)$.
By assumption (ii), for any $\gamma\in{\mathcal{U}}^{s,o}$, there exists $M\geq1$ such that (\ref{90}) holds.
Set $\delta= 1/2^{M+1}$. For any $\beta=(b_0,b_1,\cdots)\in\Sigma_{N}^{+}(A)$,
\[\rho(\alpha,\beta)<\delta\Rightarrow a_j=b_j,\; 0\leq j\leq M+1.\]
So, $x_n(\alpha),x_n(\beta)\in V_{\alpha}^{M+1,n}$, and
\[(\pi_n(\alpha),\pi_n(\beta))=(x_n(\alpha),x_n(\beta))\in V_{\alpha}^{M+1,n}\times V_{\alpha}^{M+1,n}\subset\gamma,\;n\geq0.\]
Thus, $\{\pi_n\}_{n=0}^{\infty}$ is equi-continuous at $\alpha$.
Hence, $\{\pi_n\}_{n=0}^{\infty}$ is equi-continuous in $\Sigma_{N}^{+}(A)$ by Lemma \ref{hh22}.
Further, for any $n\geq0$, $\Lambda_n$ is compact since $\Sigma_{N}^{+}(A)$ is compact and
$\Lambda_n=\pi_n(\Sigma_{N}^{+}(A))$.
Hence, $(\Sigma_{N}^{+}(A),\sigma_A)$ is topologically $\{\pi_n\}_{n=0}^{\infty}$-equi-semiconjugate
to the invariant subsystem of $(X,f_{0,\infty})$ on $\{\Lambda_n\}_{n=0}^{\infty}$.
It follows from Theorem \ref{277} that
\[h(f_{0,\infty},\Lambda_0)\leq\log\lambda(A).\]
\end{proof}

\begin{remark}
Theorems \ref{hh21} and \ref{hh2} are generalizations of Theorem 3 in \cite{Shao20} and Theorem 4.1 in \cite{Shao16}, respectively,
to a compact uniform space.
\end{remark}

With a similar proof to that of Corollary \ref{hh212}, one has the following result.

\begin{Corollary}\label{hh23}
Let all the assumptions in Theorem \ref{hh2} hold and $\bigcup_{i=1}^{N}V_i=X$ except that
assumption {\rm(i)} is replaced by (\ref{00}). Then $(\Sigma_{N}^{+}(A),\sigma_A)$ is topologically
equi-semiconjugate to $(X,f_{0,\infty})$. Consequently,
\[h(f_{0,\infty})\leq\log\lambda(A).\]
\end{Corollary}

\section*{Acknowledgement}
The author thank Professor Xiongping Dai for his useful discussions.
This research was partially supported by the Natural Science Foundation of Jiangsu Province of China (No. BK20200435)
and the Fundamental Research Funds for the Central Universities (No. NS2021053).

\end{document}